\documentclass[12pt]{amsart}

\newtheorem{theorem}{\indent\sc Theorem}[subsection]
\newtheorem*{theorem2}{\indent\sc Theorem}
\newtheorem{lemma}[theorem]{\indent\sc Lemma}
\newtheorem{proposition}[theorem]{\indent\sc Proposition}

\newenvironment{remark}{\indent\textsc{Remark.}}{}

\usepackage{latexsym,amssymb}
\usepackage{graphicx}
\usepackage{hyperref}
\usepackage{amsmath}
\usepackage{mathrsfs}
\usepackage{xkeyval}
\usepackage{rsfso}
\usepackage{tipa}

\usepackage{tikz}
\usetikzlibrary{shapes.geometric}
\usetikzlibrary{arrows}
\usetikzlibrary{decorations.markings}
\usetikzlibrary{arrows.meta}

\usepackage{a4wide}

\usepackage{stmaryrd}

\usepackage{color}

\usepackage{multirow}
\usepackage{mathtools}

\usepackage{listings}

\newcommand{\R}{\mathbb{R}}
\newcommand{\C}{\mathbb{C}}
\newcommand{\Q}{\mathbb{Q}}
\newcommand{\Z}{\mathbb{Z}}
\newcommand{\N}{\mathbb{N}}
\newcommand{\QFT}{\Q \textnormal{-factorial terminalisation}}
\newcommand{\QFl}{\Q \textnormal{-factorial}}
\newcommand{\yogh}{\textctyogh}
\newcommand{\yoghsp}{\yogh$ $ }
\newcommand{\myogh}{\textnormal{\yogh}}
\newcommand{\bihom}{\textnormal{bihom}}

\begin{document}

\title{A new addition to the zoo of isolated symplectic singularities}

\author[C.~Berry]{Callum Berry}

\begin{abstract}
We give details of a new isolated symplectic singularity found in an affine chart in a crepant partial resolution of $\C^4/G_5$,
which is 4-dimensional, isolated, and locally simply-connected.
We distinguish the new singularity among all known such by the fact that the projective tangent cone at the singularity is non-reduced.
We also find all 12 of its $\QFl$ terminalisations, in the process finding 24 for the quotient singularity $\C^4/G_5$.
\end{abstract}

\maketitle

\section{Introduction}

In recent years there has been substantial interest in symplectic singularities, as introduced by Beauville (\cite{Beauville}).
This is particularly motivated by connections with representation theory and mathematical physics.

Basic examples of isolated symplectic singularities are the rational double points
and the closures of minimal nilpotent orbits of simple complex Lie algebras.
Additionally, given a symplectic singularity $X$ and a finite group $\Gamma$ of automorphisms preserving the symplectic form on the smooth locus $X_{reg}$,
the quotient $X/\Gamma$ is also a symplectic singularity.
Hence, quotients of $\C^{2n}$ by finite subgroups of $\text{Sp}_{2n}(\C)$ that act freely away the origin,
and quotients of isolated symplectic singularities by finite groups that act freely on the smooth locus and preserve its symplectic form,
are also isolated symplectic singularities.

The main result of Beauville's seminal paper is a partial classification result:
\begin{theorem2}
Let $(V,o)$ be an isolated symplectic singularity, whose projective tangent cone is smooth.
Then $(V,o)$ is analytically isomorphic to the closure of the minimal nilpotent orbit of some simple complex Lie algebra.
\end{theorem2}

Further examples of isolated symplectic singularities have since been found.
The quasi-minimal singularities $ac_n$, $ag_2$ and $cg_2$ were found as elementary slices of affine Grassmannians in 2003 by Malkin-Ostrik-Vybornov (\cite[\S 6.3]{QuasiMinOrig}).
An infinite family $\{\mathcal{Y}(d) : d \geq 4\}$ derived from quotients of $\C^4$ by dihedral groups was found in 2021 by Bellamy-Bonnafé-Fu-Juteau-Levy-Sommers (\cite{LevyEtAl}).
An infinite collection of singularities $Y(A,0)$ given with a toric hyper-Kähler construction was found in 2023 by Namikawa (\cite{Namikawa}), which subsumes the quasi-minimal singularities (Appendix \ref{h_2c3} for $ag_2$).

It follows from the remarks in \cite[\S 4.1]{Beauville} and the fact that the local fundamental group is finite \cite{LocFdGpFin},
that any isolated symplectic singularity with non-trivial local fundamental group can be realised as the quotient of a symplectic singularity by a finite group.
As remarked by Beauville, this reduces the classification of isolated symplectic singularities to those which are locally simply-connected.

In the years since Beauville's paper,
a different approach to classification has emerged,
given by the minimal model program.
This focuses on $\QFl$ terminal singularities,
and aims to classify the maximal crepant partial resolutions of singularities (not necessarily unique).
However, as remarked in section \ref{QFT}, these two perspectives cannot be considered simultaneously.

The simplest singularities to consider would be those of the lowest dimension.
The 2-dimensional case is precisely the rational double points,
which are well known,
so we consider the 4-dimensional case instead.
Notably, a symplectic singularity has terminal singularities precisely when the codimension of the singular locus is at least 4 (\cite{NamNote}), so these isolated singularities will be terminal.

We consider the areas where one would look for new examples.
There is a rich source from the Coulomb branches of 3d $\mathcal{N}=4$ gauge theories (\cite{CoulSymp}).
Those of dimension 4 with unitary quivers have been classified (\cite{CoulClass1}),
and those of arbitrary dimension with gauge group a torus have also (\cite{CoulClass2}).
One could also look at the crepant partial resolutions of other symplectic singularities (\cite{PartResoln}).
This is motivated by the success in the case of the $\mathcal{Y}(d)$.

We follow a similar approach in this paper.
It is known that $\C^4/G_4$ has two symplectic resolutions
(and no new isolated singularities appear),
so we consider the next exceptional Shephard-Todd group, $G_5$.
The calculations are much more involved but it is worthwhile.
We find that an affine chart in the blow-up at the reduced singular locus of $\C^4/G_5$ is a new isolated symplectic singularity, which we denote \yogh.

The first endeavour of the paper is to outline the calculations required to obtain \yogh.
We then prove \yoghsp is locally simply-connected, along with detailing some useful symmetries and the Hilbert series.
The main result of the paper is then proven: that \yoghsp is distinct from all other currently known symplectic singularities.
Specifically, it has a projective tangent cone at the singularity which is non-reduced,
whilst all other known locally simply-connected, isolated symplectic singularities have a projective tangent cone at the singularity which is reduced.
Finally we give all 12 of the $\QFT$s of \yogh, allowing us to describe 24 of the 92 $\QFT$s of $\C^4/G_5$.

\subsection{Acknowledgements}

The author thanks Paul Levy for his supervision and assistance,
and the EPSRC who have funded the PhD research of which this paper forms part.
The author also thanks Daniel Juteau and Antoine Bourget for useful discussions.

\section{Algebraic Preliminaries}

Let $G$ be a complex reflection group $G \subseteq \text{GL}_2(\C)$,
acting on $V = \C^2$ and hence on the cotangent space $W = T^*V = V \oplus V^* \cong \C^4$,
preserving the natural symplectic form.
The complex reflections of $V$ in $G$ act on $W$ as symplectic reflections, fixing a 2-dimensional subspace.
As they generate $G$, this makes $G$ a symplectic reflection group of $W$, as a subgroup of $\text{Sp}_4(\C)$.

Using $G_n$ to denote the $n$-th exceptional complex reflection group in the Shepherd-Todd classification (\cite{ShepTodd}),
the first to consider is $G_4$.
Lehn and Sorger (\cite{LehnSorger}) have constructed two symplectic resolutions of $C^4/G_4$.
These were obtained by blowing up at a space containing the singular locus,
obtaining non-isolated singularities, which resolve upon a further blow-up.
There is one further partial resolution not described,
and it gives only an $a_2$ singularity (\cite{MyThesis}).
Therefore none of the partial resolutions involved yield any new isolated symplectic singularities.

The next natural choice is to investigate $\C^4/G_5$.
We realise $G_5$ as a complex reflection group in $\text{GL}_2(\C)$, as follows:
$$
a = \begin{pmatrix}
i & 0 \\
0 & -i
\end{pmatrix} \kern-2pt, \kern4pt
b = \begin{pmatrix}
0 & 1 \\
-1 & 0
\end{pmatrix} \kern-2pt, \kern4pt
c = \frac{\omega}{\sqrt{2}} \begin{pmatrix}
1 & -1 \\
i & i
\end{pmatrix} \kern-2pt, \kern4pt
d = \begin{pmatrix}
\sigma & 0 \\
0 & \sigma
\end{pmatrix} \kern-2pt,
$$
with $\omega = e^{\frac{\pi i}{12}}, \sigma = e^{\frac{2 \pi i}{3}} = \omega^8$.
Then
$$G_4 = \langle a, b, c \mid a^4 = e, b^2 = a^2, ba = a^{-1}b, c^3 = e, ca = bc, cb = a^3bc \rangle,$$
and $G_5 = \langle G_4, d \rangle$ is an extension of this by a central cyclic subgroup of order 3.

To find the invariants of the coordinate algebra of $W = \C^4$,
note that $G_5$ is generated by $a,b,c$ and $\overline{c} = acd^2$.
Consider the eigenvectors in $\C[W]^{\langle a,b \rangle}$ of $c$ and $\overline{c}$.
Using overline to represent complex conjugation of polynomials, we have the eigenvectors as follows:

$$
\begin{array}{rcl}
F_{1,1} &\kern-5pt = &\kern-5pt xX + yY, \\
F_{4,0} &\kern-5pt = &\kern-5pt x^4 + y^4 - 2i\sqrt{3}x^2y^2, \\
F_{3,1} &\kern-5pt = &\kern-5pt x^3Y - i\sqrt{3}x^2yX + i\sqrt{3}xy^2Y - y^3X, \\
G_{2,2} &\kern-5pt = &\kern-5pt x^2X^2 + i\sqrt{3}x^2Y^2 + i\sqrt{3}y^2X^2 + y^2Y^2 - 4xyXY, \\
F_{1,3} &\kern-5pt = &\kern-5pt xY^3 + i\sqrt{3}yXY^2 - i\sqrt{3}xX^2Y - yX^3, \\
F_{0,4} &\kern-5pt = &\kern-5pt X^4 + Y^4 + 2i\sqrt{3}X^2Y^2, \\
F_{6,0} &\kern-5pt = &\kern-5pt x^5y - xy^5, \\
G_{5,1} &\kern-5pt = &\kern-5pt x^5X - 5x^4yY - 5xy^4X + y^5Y, \\
G_{4,2} &\kern-5pt = &\kern-5pt x^4XY - 2x^3yY^2 + 2xy^3X^2 - y^4XY, \\
F_{3,3} &\kern-5pt = &\kern-5pt x^3XY^2 - x^2yY^3 - xy^2X^3 + y^3X^2Y, \\
G_{2,4} &\kern-5pt = &\kern-5pt 2x^2XY^3 + xyX^4 - xyY^4 - 2y^2X^3Y, \\
G_{1,5} &\kern-5pt = &\kern-5pt xX^5 - 5xXY^4 - 5yX^4Y + yY^5 \\
F_{0,6} &\kern-5pt = &\kern-5pt X^5Y - XY^5,
\end{array}
$$
with the eigenvalues as follows:
\renewcommand{\arraystretch}{1.2}
\begin{center}
\begin{tabular}{| c | c | c | c | c |}
\cline{3-5}
\multicolumn{2}{c|}{\multirow{2}{*}{}} & \multicolumn{3}{|c|}{$c$}
\\ \cline{3-5}
\multicolumn{2}{c|}{} & $1$ & $\sigma$ & $\sigma^2$
\\ \hline
\multirow{4}{*}{$\overline{c}$} &
\multirow{2}{*}{$1$} &
$F_{1,1}, F_{6,0},$ &
\multirow{2}{*}{$\overline{F_{4,0}}, \overline{F_{1,3}}$} &
\multirow{2}{*}{$\overline{F_{3,1}}, \overline{F_{0,4}}$}
\\
& & $F_{3,3}, F_{0,6}$ & &

\\ \cline{2-5}
& $\sigma$ &
$F_{4,0}, F_{1,3}$ &
$G_{4,2}, G_{1,5}$ &
$\overline{G_{2,2}}$
\\ \cline{2-5}
& $\sigma^2$ &
$F_{3,1}, F_{0,4}$ &
$G_{2,2}$ &
$G_{5,1}, G_{2,4}$
\\ \hline
\end{tabular}
\end{center}
\renewcommand{\arraystretch}{1}

We can manipulate this to find the coordinate algebra of $W/G_5$;
$$\C[W/G_5] = \C[h, A_1, A_0, A_{-1}, B_1, B_0, B_{-1},
C_2, C_1, C_0, C_{-1}, C_{-2}],$$
with the generators given as

$$
\begin{array}{c}
h = F_{1,1}, \kern4pt
A_1 = 6F_{6,0}, \kern4pt
A_0 = 3F_{3,3}, \kern4pt
A_{-1} = 6F_{0,6}, \\
B_1 = -F_{4,0}F_{3,1} - \overline{F_{4,0}} \kern1pt \overline{F_{3,1}}, \kern4pt
B_0 = F_{3,1}F_{1,3} + \overline{F_{3,1}} \kern1pt \overline{F_{1,3}}, \kern4pt
B_{-1} = -F_{1,3}F_{0,4} - \overline{F_{1,3}} \kern1pt \overline{F_{0,4}}, \\
C_2 = {F_{4,0}}^3 + \overline{F_{4,0}}^3, \kern4pt
C_1 = -{F_{3,1}}^3 - \overline{F_{3,1}}^3, \kern4pt
C_0 = F_{4,0}{F_{1,3}}^2 + \overline{F_{4,0}} \kern1pt \overline{F_{1,3}}^2 + 4 {F_{1,1}}^3F_{3,3}, \\
C_{-1} = -{F_{1,3}}^3 - \overline{F_{1,3}}^3, \kern4pt
C_{-2} = {F_{0,4}}^3 + \overline{F_{0,4}}^3
\end{array}
$$
These generators have 35 relations, which have been listed in Appendix \ref{G5Relns}.
The generators as polynomials and the 35 relations are all real,
which gives that $W / G_5$ is split over $\R$.

We have an index-swapping symmetry $\phi$ coming from $\C[x,y,X,Y]$, given as $$\phi: x \mapsto X, y \mapsto Y, X \mapsto x, Y \mapsto y.$$
This occurs because $A \in \text{GL}_2(\C)$ acting on $V = \C^2$ acts as 
$ \begin{pmatrix}
A & 0 \\
0 & A^{-T}
\end{pmatrix} $
on the cotangent space $T^*V \cong \C^4$,
and $G_5$ is stable under inverse transpose.
$\phi$ belongs to GSp$_4(\C)$ and multiplies the symplectic form by $-1$.
It is an order 2 symmetry, which acts on the generators as
{$$
\begin{array}{c}
\phi: h \mapsto h, A_1 \mapsto A_{-1}, A_0 \mapsto -A_0, B_1 \mapsto -B_{-1}, B_0 \mapsto B_0, \\ C_2 \mapsto C_{-2}, C_1 \mapsto -C_{-1}, C_0 \mapsto C_0.
\end{array}
$$}

We have a useful projection $\rho$ onto $\C[x,y]$, given as
$$\rho: x \mapsto x, y \mapsto y, X \mapsto y, Y \mapsto -x.$$
Consider the bidegree given by specifying $x, y$ as degree $(1,0)$ and $X, Y$ as degree $(0,1)$.
Most of the polynomials considered are homogeneous to both coordinates of this bidegree.
Call this property being bi-homogeneous.
It can be proven that in $\C[x,y,X,Y]$,
the bi-homogeneous part of $\ker(\rho)$ equals the bi-homogeneous part of $\langle h \rangle$.
This can be used to match algebraic expressions in a given bidegree by their value under $\rho$, to show the difference is divisible by $h$.

Recall that the singular locus in $W/G_5$ is the image of the set of points in $W$ with non-trivial stabiliser in $G_5$.
The only non-identity elements of $G_5$ that can fix points in $W$ other than the origin are the symplectic reflections,
so the singular locus is the union of the images of the fixed point subspaces of the symplectic reflections of $G_5$.
Calculating this in Magma gives that the singular locus has radical ideal
$$
I := \langle R_{4,4}, R_{9,3}, {R_{9,3}}', R_{6,6}, {R_{6,6}}', R_{3,9}, {R_{3,9}}', R_{13,1}, R_{1,13}, R_{14,2}, R_{2,14}, R_{19,1}, R_{1,19}, R_{24,0}, R_{0,24} \rangle
$$
with the generators having expressions in the semi-invariants found earlier:

$$
\begin{array}{lll}
R_{4,4} := G_{2,2} \overline{G_{2,2}}, &
R_{9,3} := F_{3,1} \overline{F_{4,0}} \kern1pt \overline{G_{2,2}}, &
{R_{9,3}}' := F_{4,0} \overline{F_{3,1}} G_{2,2}, \\
R_{6,6} := F_{0,4} \overline{F_{4,0}} \kern1pt \overline{G_{2,2}}, &
{R_{6,6}}' := F_{4,0} \overline{F_{0,4}} G_{2,2}, &
R_{3,9} := F_{0,4} \overline{F_{1,3}} \kern1pt \overline{G_{2,2}}, \\
{R_{3,9}}' := F_{1,3} \overline{F_{0,4}} G_{2,2}, &
R_{13,1} := F_{4,0} \overline{F_{4,0}} G_{5,1}, &
R_{1,13} := F_{0,4} \overline{F_{0,4}} G_{1,5}, \\
R_{14,2} := F_{4,0} F_{3,1} \overline{F_{4,0}} \kern1pt \overline{F_{3,1}}, &
R_{2,14} := F_{1,3} F_{0,4} \overline{F_{1,3}} \kern1pt \overline{F_{0,4}}, &
R_{19,1} := F_{4,0} F_{3,1} \overline{F_{4,0}}^3, \\
R_{1,19} := F_{1,3} F_{0,4} \overline{F_{0,4}}^3, &
R_{24,0} := {F_{4,0}}^3 \overline{F_{4,0}}^3, &
R_{0,24} := {F_{0,4}}^3 \overline{F_{0,4}}^3, \\
\end{array}
$$

\section{Obtaining and Verifying a New Symplectic Singularity}

In this section we give the details of blowing up $\C^4/G_5$ at the reduced singular locus,
and the subsequent description of the new singularity \yogh.
We also give the verification of \yogh$ $ as an isolated symplectic singularity.

Blowing up at $I$, we can cover the blow-up with affine charts corresponding to the generators of the ideal.
Label the blow-up as $\pi_I: S_I \rightarrow \C^4/G_5$.
Amongst the affine charts, we can show that those corresponding to generators
$$R_{4,4}, R_{9,3}, R_{3,9}, R_{13,1}, R_{1,13}, R_{24,0}, R_{0,24}$$
cover $S_I$,
so considering these will be sufficient for our purposes.
Label these affine charts as 
$$X_{4,4}, X_{9,3}, X_{3,9}, X_{13,1}, X_{13,1}, X_{24,0}, X_{0,24}$$
respectively,
i.e. $X_{4,4}$ is the affine open subset of $S_I$ consisting of the points with the projective coordinate corresponding to $R_{4,4}$ being non-zero.
Using the symmetry $\phi$, we only need to calculate
$X_{4,4}, X_{9,3}, X_{13,1}$ and $X_{24,0}$
to check for new singularities.

\subsection{Affine charts of the blow-up}

\subsubsection{\texorpdfstring{$X_{24,0}$}{X_24,0}}

$$\C[X_{24,0}] = \C \left[ h, A_1, \ldots, C_{-2},
\frac{R_{4,4}}{R_{24,0}}, \frac{R_{9,3}}{R_{24,0}}, \ldots \frac{R_{0,24}}{R_{24,0}} \right] $$
We can simplify this by eliminating some of these generators.
For example, we can examine potential expressions of these generators in the bidegree $(1,1)$ to find
$$ h R_{24,0} = \frac{1}{6} i\sqrt{3} {A_1}^2 R_{13,1}
+ A_1 R_{19,1} + \frac{1}{2} C_2 R_{13,1}
$$
and dividing through by $R_{24,0}$ gives an equation to show that $h$ is redundant as a generator.
Similarly, we can examine potential expressions in the bidegree $(-10,2)$ to find
$$ R_{14,2} R_{24,0} = \frac{1}{3} i\sqrt{3} A_1 R_{13,1} R_{19,1} + {R_{19,1}}^2
$$
and dividing through by ${R_{24,0}}^2$ gives an equation to show that $\frac{R_{14,2}}{R_{24,0}}$ is redundant.
Proceeding in this manner, we can reduce the generating set of this coordinate algebra to 4 generators;
$\C[X_{24,0}] = \C[A_1, C_2,
\frac{R_{13,1}}{R_{24,0}},
\frac{R_{19,1}}{R_{24,0}} ] $.
Since $X_{24,0}$ is 4-dimensional, this must mean that this is a set of free generators for its coordinate algebra, and so $X_{24,0} \cong \C^4$ is smooth.

\subsubsection{\texorpdfstring{$X_{13,1}$}{X_13,1}}

Label the generators of $\C[X_{13,1}]$ as 
$f_{i,j} := \frac{R_{i,j}}{R_{13,1}},
f_{i,j}' := \frac{{R_{i,j}}'}{R_{13,1}}
$.
We can reduce the generating set of the coordinate algebra to 8 generators;
$$
\C[X_{13,1}] = \C[
h, A_1, f_{4,4}, f_{9,3}, f_{9,3}', f_{14,2}, f_{19,1}, f_{24,0}
]. $$

We claim that $X_{13,1}$ is isomorphic to the minimal nilpotent orbit closure in $\mathfrak{sl}_3$, denoted $a_2$.
It is well-known that the $2 \kern-2pt \times \kern-2pt 2$ minors generate the ideal in $\C[\mathfrak{sl}_3]$ vanishing on the minimal nilpotent orbit.
Hence, for $a_2 = \overline{\mathcal{O}_{\min}(\mathfrak{sl}_3)}$,
we can show that the map from $\C[a_2]$ into $\C[X_{13,1}]$
represented by the matrix
$$\begin{array}{c}
\begin{pmatrix}
-h - i \sqrt{3} f_{14,2} &
2i \sqrt{3} f_{9,3}' & f_{4,4} \\
-f_{19,1} & 2i \sqrt{3} f_{14,2} & f_{9,3} \\
-f_{24,0} &
2i \sqrt{3} f_{19,1} - 2 A_1 &
h - i \sqrt{3} f_{14,2}
\end{pmatrix},
\\
\textnormal{which can also be written as }
\frac{1}{R_{13,1}}
\begin{pmatrix}
G_{2,2} \\ F_{3,1} \overline{F_{4,0}} \\ {F_{4,0}}^2 \overline{F_{4,0}}
\end{pmatrix}
\begin{pmatrix}
-F_{4,0} \overline{F_{4,0}}^2 &
2i\sqrt{3} F_{4,0} \overline{F_{3,1}} & 
\overline{G_{2,2}}
\end{pmatrix},
\end{array}
$$
gives an isomorphism between $X_{13,1}$ and $a_2$.
This shows that $X_{13,1}$ is smooth except for a single isolated point,
which is not in any of the other considered $X_{i,j}$ and is of type $a_2$.

\subsubsection{\texorpdfstring{$X_{9,3}$}{X_9,3}}

Label the generators of $\C[X_{9,3}]$ as 
$g_{i,j} := \frac{R_{i,j}}{R_{9,3}},
g_{i,j}' := \frac{{R_{i,j}}'}{R_{9,3}}
$.
Define
$$
\begin{array}{ccccc}
a := A_1 - i\sqrt{3} h g_{14,2}, &
b := g_{4,4}, &
c := 3g_{6,6} - 3 h^2 g_{4,4}, &
d := g_{13,1}, &
e := g_{14,2}.
\end{array}
$$
We can reduce the generating set of the coordinate algebra to 5 generators;
$ \C[X_{9,3}] = \C[a,b,c,d,e]. $
Among the remaining 5 generators, there is 1 relation;
$$
abd(1 + 2be) + cd^2(-1 + 2be) + 6e(1 + be) = 0
$$
It is straightforward to check that this is smooth.

\subsubsection{\texorpdfstring{$X_{4,4}$}{X_4,4}}

Label the generators of $ R := \C[X_{4,4}]$ as
$k_{i,j} := \frac{{R_{i,j}}}{{R_{4,4}}},
k_{i,j}' := \frac{{R_{i,j}}'}{{R_{4,4}}}
$.
Reparametrise the generators as follows:
$$
\begin{array}{ccc}
a_1 := k_{3,9} + k_{3,9}', &
a_0 := -k_{6,6} - k_{6,6}' + 2h^2, &
a_{-1} := k_{9,3} + k_{9,3}', \\
b_1 := \frac{1}{3} i\sqrt{3} (-k_{3,9} + k_{3,9}'), &
b_0 := \frac{1}{3} i\sqrt{3} (k_{6,6} - k_{6,6}'), &
b_{-1} := \frac{1}{3} i\sqrt{3} (-k_{9,3} + k_{9,3}')
\end{array}
$$
We then use the following list of 12 generators:
$$
\begin{array}{c}
\begin{array}{ccc}
h, &
c_i := a_i + b_i, &
d_i := a_i - b_i,
\end{array}
\\
\begin{array}{ccccc}
z_2 := -2k_{1,13}, &
z_1 := \frac{2}{3} A_{-1} - 2hb_1, &
z_0 := \frac{2}{3} A_0, &
z_{-1} := \frac{2}{3} A_1 + 2hb_{-1}, &
z_{-2} := 2k_{13,1}
\end{array}
\end{array}
$$
That is, the coordinate algebra is generated over $\C$ as
$$R = \C[h, c_1, c_0, c_{-1}, d_1, d_0, d_{-1},
z_2, z_1, z_0, z_{-1}, z_{-2}].$$
This will be shown to be singular only at the origin.
Label this singularity as \yogh.

\subsection{Coordinate algebra}

Among these 12 generators of $\C[\textnormal{\textctyogh}]$ there are 35 relations;
$$
\begin{array}{cc}
c_1d_0 - c_0d_1 - 2hz_1 = 0 &
c_{-1}d_0 - c_0d_{-1} + 2hz_{-1} = 0 \\
c_1d_{-1} - c_0d_0 - 2hz_0 = 0 &
c_{-1}d_1 - c_0d_0 + 2hz_0 = 0 \\
c_1c_{-1} - {c_0}^2 - 2h^2d_0 = 0 &
d_1d_{-1} - {d_0}^2 - 2h^2c_0 = 0 \\
c_1z_1 - c_0z_2 + h{d_1}^2 = 0 &
d_1z_1 - d_0z_2 - h{c_1}^2 = 0 \\
c_{-1}z_{-1} - c_0z_{-2} - h{d_{-1}}^2 = 0 &
d_{-1}z_{-1} - d_0z_{-2} + h{c_{-1}}^2 = 0 \\
c_1z_0 - c_0z_1 + hd_1d_0 = 0 &
d_1z_0 - d_0z_1 - hc_1c_0 = 0 \\
c_{-1}z_0 - c_0z_{-1} - hd_{-1}d_0 = 0 &
d_{-1}z_0 - d_0z_{-1} + hc_{-1}c_0 = 0 \\
c_0z_1 - c_{-1}z_2 - 2h^3c_1 + hd_1d_0 = 0 &
d_0z_1 - d_{-1}z_2 + 2h^3d_1 - hc_1c_0 = 0 \\
c_0z_{-1} - c_1z_{-2} + 2h^3c_{-1} - hd_{-1}d_0 = 0 &
d_0z_{-1} - d_1z_{-2} - 2h^3d_{-1} + hc_{-1}c_0 = 0 \\
c_1z_{-1} - c_0z_0 + h{d_0}^2 = 0 &
d_1z_{-1} - d_0z_0 - h{c_0}^2 = 0 \\
c_{-1}z_1 - c_0z_0 - h{d_0}^2 = 0 &
d_{-1}z_1 - d_0z_0 + h{c_0}^2 = 0 \\
z_2z_0 - {z_1}^2 - h^2c_1d_1 = 0 &
z_{-2}z_0 - {z_{-1}}^2 - h^2c_{-1}d_{-1} = 0 \\
z_2z_{-1} - z_1z_0 - h^2c_1d_0 - h^2c_0d_1 = 0 &
z_{-2}z_1 - z_{-1}z_0 - h^2c_{-1}d_0 - h^2c_0d_{-1} = 0 \\
z_2z_{-2} - z_0^2 + 4h^6 - 4h^2c_0d_0 = 0 &
z_1z_{-1} - z_0^2 - h^2c_0d_0 = 0 \\
{c_1}^3 + {d_1}^3 + 2z_2z_1 = 0 &
{c_{-1}}^3 + {d_{-1}}^3 + 2z_{-2}z_{-1} = 0 \\
{c_1}^2c_0 + {d_1}^2d_0 + 2{z_1}^2 = 0 &
{c_{-1}}^2c_0 + {d_{-1}}^2d_0 + 2{z_{-1}}^2 = 0 \\
c_1{c_0}^2 + d_1{d_0}^2 + 2z_1z_0 = 0 &
c_{-1}{c_0}^2 + d_{-1}{d_0}^2 + 2z_{-1}z_0 = 0 \\
{c_0}^3 + {d_0}^3 + 2z_1z_{-1} = 0
\end{array}
$$

This has been verified to be a complete list by checking that the free algebra over $\Q$ with free generators given by the variables $h, \ldots, z_{-2}$ has the ideal generated by these relations as a prime ideal of dimension 4,
via computer calculation,
and then using this result to extend the method of proof from base field $\Q$ to $\C$.

The primality test used by computers uses elimination theory,
which can be extended to a higher field in our case,
provided the variable $x$ chosen to be eliminated at each successive step shows up in some relation in the form
$xy_i + f(y_1, \ldots, y_n) = 0$.
This is easy to confirm for this case, and all other cases in this paper.

A conceptual proof that	this ideal is prime, using the projection $\rho$ discussed earlier, is given in Appendix \ref{rhoAppendix}.

\subsection{Isolatedness}

It is clear from the set of relations that the origin is a singular point in \yogh.
To show that it is the only singular point,
we take advantage of the other affine charts we calculated.

Recall that $\C^4 / G_5$ is split over $\R$.
As the the ideal being blown up is stable under complex conjugation (acting as an $\R$-symmetry),
we know the affine charts $X_{9,3}' := \{R_{9,3}' = \overline{R_{9,3}} \neq 0 \}$ and $X_{3,9}' := \{R_{3,9}' = \overline{R_{3,9}} \neq 0 \}$ are symmetric to $X_{9,3}$ and $X_{3,9}$ respectively.
In particular they are also smooth.

To summarise, we know that $X_{24,0}, X_{0,24}, X_{9,3}$, $X_{9,3}'$, $X_{3,9}$ and $X_{3,9}'$ are smooth,
and that the only singularities in $X_{13,1}$ and $X_{1,13}$ are their respective origins.
Due to $f_{4,4}$ appearing as a generator in the coordinate algebra of $X_{13,1}$, and the symmetric result for $X_{1,13}$,
the origins for these affine charts lie outside of $X_{4,4}$.
Therefore, it suffices to show that $X_{4,4}$ without its origin is covered by its intersections with these other affine charts.
To prove this, it suffices to show that in the free algebra generated by $h, c_1, \ldots, z_{-2}$, all of the generators of $\C[\C^4/G_5]$ lie in the radical of the ideal which is generated by the 35 relations and the variables
$k_{9,3}, {k_{9,3}}', k_{3,9}, {k_{3,9}}', k_{13,1}, k_{1,13}, k_{24,0}, k_{0,24}$.
This can be verified via computer calculation.

\subsection{Normality}

To show that \textctyogh$ $ is a symplectic singularity, we first need to show that it is normal.
We show it is Cohen-Macaulay.
Combined with the fact that \yogh$ $ is nonsingular in codimension 1,
this implies normality via Serre's criterion.
Therefore, it remains to show that the local ring of \textctyogh$ $ at the origin is Cohen-Macaulay.
It suffices to show that the local ring has depth 4, i.e. a regular sequence of length 4 in the maximal ideal.
We can verify using a computer that $h, z_2, c_{-1}, c_1 + z_{-2}$ is a regular sequence of length 4 in $R$ that is contained in the maximal ideal corresponding to the origin.
This gives that it is the same in the localisation at the origin, proving that \textctyogh$ $ is Cohen-Macaulay, and so by extension normal.

\subsection{Symplectic singularity}

Recall the blow-up of $\C^4/G_5$ at $I$, $\pi_I: S_I \rightarrow \C^4/G_5$.
As a result of being defined in such a manner, $S_I$ has the property that any of its resolutions will compose to a resolution of $\C^4/G_5$.
Hence, by the definition of symplectic singularity,
$S_I$ can only fail to be a symplectic singularity if the extension of the pull-back of the symplectic form $\omega_0$ of $\C^4/G_5$ to $S_I$ is degenerate on the smooth locus ${S_I}^{sm}$.
We prove this by relating it to the condition of $\pi_I$ being semi-small.

Degeneracy of the symplectic form $\omega_I$ is equivalent to $\omega_I \wedge \omega_I = 0$, and so the subvariety it defines will either be empty or codimension 1.
Label this subvariety $T$.

Stratify $\C^4/G_5$ into its smooth locus, the singular points away from the origin, and the origin.
The blow-up $\pi_I$ acts as an isomorphism on the smooth locus, and so these points cannot lie in $T$.
The singular points away from the origin are locally isomorphic to $\C^2 \times A_2$, and so $\pi_I$ must act as a symplectic resolution, and so the preimages of these points cannot lie in $T$.

This leaves the preimage of the origin under the blow-up.
To show $\dim({\pi_I}^{-1}(0)) \leq 2$,
we describe its intersection with each of the affine charts in the cover as contained in a subvariety of dimension $2$.
These subvarieties are given by the ideals
$$
\langle A_1, C_2 \rangle \subseteq X_{24,0},
\langle h, A_1 \rangle \subseteq X_{13,1},
\langle h, A_1 \rangle \subseteq X_{9,3},
\langle h, z_0 \rangle \subseteq X_{4,4},
$$
and similar for the corresponding symmetric affine charts.
Therefore $\dim({\pi_I}^{-1}(0)) \leq 2$, so $\dim(T) \leq 2$, so $T$ must be empty.
This shows $S_I$ is a symplectic singularity, and so also \yogh.

\section{Properties of the New Singularity}

\subsection{Useful symmetries}

There are multiple useful symmetries of \textctyogh.

We still have the index-swapping symmetry $\phi$ coming from $\C[x,y,X,Y]$, given as $$\phi: x \mapsto X, y \mapsto Y, X \mapsto x, Y \mapsto y.$$
This is of order 2 and acts on the generators as
$$
\begin{array}{c}
\phi: h \mapsto h, c_1 \mapsto -c_{-1}, c_0 \mapsto c_0,
d_1 \mapsto -d_{-1}, d_0 \mapsto d_0, 
z_2 \mapsto -z_{-2}, z_1 \mapsto z_{-1}, z_0 \mapsto -z_0.
\end{array}
$$

We have an order 2 symmetry that swaps the $c_i$ and $d_i$.
Define an automorphism $\tau$ on $\C[x,y,X,Y]$ by sending
$$\tau: x \mapsto e^{\frac{5 \pi i}{12}}x, \kern5pt
y \mapsto e^{\frac{- \pi i}{12}}y, \kern5pt
X \mapsto e^{\frac{-5 \pi i}{12}}X, \kern5pt
Y \mapsto e^{\frac{\pi i}{12}}Y.$$
The monomials of the generators of $\C[x,y,X,Y]^{G_5}$ are multiplied by $1$ or $-1$ under this.
We can check that $\tau$ acts on the singularity by fixing $h$, swapping the $c_i$'s with the $d_i$'s, and multiplying the $z_i$'s by $-1$.

We have a symmetry $\chi$ of order 3 that fixes $h$ and the $z_i$, and multiplies $c_i$ by $\sigma = e^{\frac{2\pi i}{3}}$ and $d_i$ by $\sigma^2$.
This is specific to this singularity, and does not induce an action on $\C^4/G_5$.

These symmetries form a dihedral group of order 12.
It is easy to see that $\phi$ commutes with $\tau$ and $\chi$,
whilst $\tau$ and $\chi$ do not commute.

\subsection{Local fundamental group}

We now want to determine the local fundamental group of \yoghsp around its singular point, labelled $0$ in this section.

\begin{lemma}
The local fundamental group of \kern2pt \textnormal{\yogh} at $0$ is trivial.
\end{lemma}
\begin{proof}
For the natural grading on $\C[x,y,X,Y]$,
all generators of $\C[\myogh]$ are homogeneous of positive degree,
giving a contracting $\C^{\times}$-action on \textctyogh$ $.
This gives that the local fundamental group of $(\textnormal{\textctyogh}, 0)$ is isomorphic to the fundamental group of \textctyogh $\setminus \{0\}$.
Since \yogh$ $ is irreducible, \yogh $\setminus \{0\}$ is also.
All closed strict subvarieties of this are therefore of strictly lower complex dimension,
so all open non-empty subvarieties $Z$ will have that the map $\pi_1(Z) \rightarrow \pi_1(\myogh \setminus \{0\})$ is surjective.
(\cite[Chapitre X, Théorème 2.3]{Godbillon})

This applies to the intersection of \textctyogh$ $ $ = X_{4,4}$ with any of the other $X_{i,j}$, however we take the intersection with $X_{9,3} \cup X_{13,1}$ because as we now prove, this has trivial fundamental group.

Label $U_1 := X_{9,3} \cap X_{4,4}$, $U_2 := X_{13,1} \cap X_{4,4}$, $U_3 := X_{9,3} \cap X_{13,1} \cap X_{4,4}$.
For ease of calculation, where possible we work in the coordinate algebra of $X_{9,3}$.
Consider $U_1$.
We can show that the map $\kappa_1: U_1 \rightarrow \C^\times$ given by projecting to the $g_{4,4}$ coordinate is a trivial fiber bundle, with fiber isomorphic to the space
$\C \times \{ (x,y,z) \in \C^3 \kern2pt | \kern2pt x^2 - y^2z = 1 \}$,
which has trivial fundamental group \cite[Lemma 4.1]{LevyEtAl}.
Applying the long exact sequence in homotopy shows that
$ \kappa_* : \pi_1(U_1) \rightarrow \pi_1(\C^\times)$
is an isomorphism.
This tells us that $\pi_1(U_1) \cong \Z$, and that the homotopy class of a loop in $U_1$ in this isomorphism is determined by how the $g_{4,4}$ coordinate winds around the origin of $\C^\times$.

Consider $U_2$.
We can show that the map $\kappa_2: U_2 \rightarrow \C^\times$ given by projecting to the $f_{4,4}$ coordinate is a trivial fiber bundle, with fiber isomorphic to $\C^3$.
This tells us that $\pi_1(U_2) \cong \Z$, and that the homotopy class of a loop in $U_2$ in this isomorphism is determined by how the $f_{4,4}$ coordinate winds in $\C^\times$.
For comparison in the coordinate algebra of $X_{9,3}$, note that in $U_3$, where the $g_{13,1}$ coordinate is non-zero, we have that
$$
f_{4,4} = \frac{\widetilde{R_{4,4}}}{\widetilde{R_{13,1}}}
= \frac{\widetilde{R_{4,4}}}{\widetilde{R_{9,3}}}
\frac{\widetilde{R_{9,3}}}{\widetilde{R_{13,1}}}
= g_{4,4} \kern1pt {g_{13,1}}^{-1}.
$$

Consider $U_3$.
We can show that the map $\kappa_2: U_2 \rightarrow \C^\times \times \C^\times$ given by projecting to the $g_{4,4}$ and $g_{13,1}$ coordinates is a trivial fiber bundle, with fiber isomorphic to $\C^2$.
This tells us that $\pi_1(U_3) \cong \Z \times \Z$, and that the homotopy class of a loop in $U_3$ is determined by how the $g_{4,4}$ and $g_{13,1}$ coordinates wind in $\C^\times$.

To apply the Seifert-van Kampen theorem, choose a basepoint $p \in U_3$ with $g_{4,4}$ and $g_{13,1}$ coordinates equal to 1.
In $U_1$, take a loop $\gamma_1$ based at $p$ whose $g_{4,4}$ coordinate winds around the origin in $\C^\times$ positively once, so that $[\gamma_1]$ generates $\pi_1(U_1) \cong \Z$.
In $U_2$, take a loop $\gamma_2$ based at $p$ whose $f_{4,4}$ coordinate winds in $\C^\times$ positively once, so that $[\gamma_2]$ generates $\pi_1(U_2) \cong \Z$.
In $U_3$, take a loop $\gamma_3$ based at $p$ whose $g_{4,4}$ coordinate winds in $\C^\times$ positively once, and whose $g_{13,1}$ coordinate stays constant as 1;
and take a loop $\gamma_4$ based at $p$ whose $g_{13,1}$ coordinate winds in $\C^\times$ positively once, and whose $g_{4,4}$ coordinate stays constant as 1.
Then $[\gamma_3]$ and $[\gamma_4]$ generate $\pi_1(U_3) \cong \Z^2$ freely.

With this set-up, we can determine the pushout of the diagram
$$
\pi_1(U_1) \leftarrow \pi_1(U_3) \rightarrow \pi_1(U_2).
$$
We have that 
$$
-[\gamma_1] \mapsfrom [\gamma_3] \mapsto 0
$$
and
$$
[\gamma_1] \mapsfrom [\gamma_4] \mapsto [\gamma_2],
$$
so the pushout is trivial.
Therefore, by the Seifert-van Kampen theorem, $\pi_1(X_{4,4} \cap (X_{9,3} \cup X_{13,1}))$ is trivial.
As this surjects onto $\pi_1(X_{4,4} \setminus \{ 0 \})$, this shows that the local fundamental group of $X_{4,4} = $ \textctyogh$ $ around $0$ is trivial.
\end{proof}

\subsection{Hilbert series}

Recall the natural grading on $\C[x,y,X,Y]$.
This induces a grading on $\C[\myogh]$, with $h$ of degree 2, $c_i$ and $d_i$ of degree 4, and $z_i$ of degree 6.

With this, it can be calculated (for example with Magma) that the Hilbert series is
$$
\frac{1 + 4t^4 + 4t^6 + 4t^8 + t^{12}}{(1 - t^2) (1 - t^4)^2 (1 - t^6)}.
$$

For contrast, the Hilbert series of $\C^4/G_5$ is 
$$
\frac{1 - t^4 + 2t^6 + 3t^8 - 2t^{10} + 3t^{12} + 2t^{14} - t^{16} + t^{20}}
{(1 - t^2) (1 - t^4) (1 - t^6) (1 - t^{12})}
$$

\section{Projective Tangent Cones of Isolated Symplectic Singularities}

By the results of the previous section, \yoghsp is a locally simply-connected, isolated, 4-dimensional symplectic singularity.
In this section we prove it is not (analytically) isomorphic to any previously known isolated 4-dimensional symplectic singularity.
Recalling the classification of these known cases into families in the introduction,
we can dismiss the quotient singularities due to \yoghsp being locally simply-connected.
This leaves the closures of minimal nilpotent orbits in simple Lie algebras,
the infinite family obtained from dihedral groups,
and a specific construction of toric hyper-Kähler varieties.

The key property we use to distinguish \yoghsp is that its projective tangent cone at the singular point is non-reduced.
Given the coordinate algebra $R = \C[\myogh]$ and the maximal ideal $\mathfrak{m} = \langle h, c_1, \ldots, z_{-2} \rangle$ at the singularity $0$,
recall that the projective tangent cone of \yoghsp at the singularity, $\mathbb{P}T_0(\myogh)$,
is given as the Proj of the associated graded algebra
$\textnormal{gr}_{\mathfrak{m}}(R) = \bigoplus_{i \geq 0}
\frac{\mathfrak{m}^i}{\mathfrak{m}^{i+1}}$.

All of these terms $\frac{\mathfrak{m}^i}{\mathfrak{m}^{i+1}}$ are invariant under localisation at $\mathfrak{m}$, and then under completion.
Hence, any analytic isomorphism between varieties $(X, x)$ and $(Y, y)$
clearly induces an isomorphism on the projective tangent cones, $\mathbb{P}T_x(X) \cong \mathbb{P}T_y(Y)$.
Therefore, the projective tangent cone at the singularity is indeed an invariant of an isolated symplectic singularity, and can be used to distinguish it from others.

As an associated graded algebra of $R$, we have $ \dim(\textnormal{gr}_{\mathfrak{m}}(R)) = 4$.
Therefore, the projective tangent cone Proj$(\textnormal{gr}_{\mathfrak{m}}(R))$ is reduced precisely when $\textnormal{gr}_{\mathfrak{m}}(R)$ is reduced as a ring,
and likewise Proj$(\textnormal{gr}_{\mathfrak{m}}(R))$ is in addition irreducible precisely when $\textnormal{gr}_{\mathfrak{m}}(R)$ is an integral domain.

Consider $R = \C[\textnormal{\yogh}]$,
with maximal ideal $\mathfrak{m} = \langle h, c_1, \ldots, z_{-2} \rangle$,
and coordinate algebra of affine tangent cone at the singularity $S = \textnormal{gr}_{\mathfrak{m}}(R)$.
Observing the relations for \yogh, we can see that the relation
$
{c_1}^2c_0 + {d_1}^2d_0 + 2{z_1}^2 = 0
$
in $R$ becomes the relation
$
{z_1}^2 = 0
$
in $S$.
As there are no linear terms in the relation list given,
it is clear that $z_1 \neq 0$ in $S$, so this is a nilpotent element.
Therefore $S$ is non-reduced,
so $\mathbb{P}T_0(\myogh)$ is non-reduced.

This is a property shared by many quotient singularities, for example the rational double points of types $D$ and $E$.
However, it appears to be exceptional for the case of being locally simply-connected.

For an isolated symplectic singularity $X$ with sole singular point $0$, the projective tangent cone $\mathbb{P}T_0(X)$ is smooth precisely for the case of closures of minimal nilpotent orbits of simple Lie algebras (\cite[p.1]{Beauville}).
Being smooth implies being reduced, so this case is covered.

\subsection{The dihedral family \texorpdfstring{$\mathcal{Y}$}{Y}(d)}

Consider the family of symplectic singularities $\mathcal{Y}(d)$ (\cite{LevyEtAl}).
The coordinate algebra $R_d = \C[\mathcal{Y}(d)]$ is $\C[q,Q,e,b_0, b_1, \ldots, b_d]$ with relations of the form
$$
\left\lbrace \begin{array}{l}
eb_j = qb_{j+1} + Qb_{j-1} \\
b_j b_k - b_{j-1} b_{k+1} \in \textnormal{span}_\C(\{q,Q,e\}^{d-2})
\end{array} \right.
\textnormal{for } 1 \leq j \leq k \leq d-1.
$$

The case $d = 4$ is the coordinate algebra of the $a_2$ singularity.
This is the closure of the minimal nilpotent orbit for $\mathfrak{sl}_3$,
and so the projective tangent cone is smooth,
notably also reduced.

Consider $d > 4$.
With $\mathfrak{m} = \langle q,Q,e, b_0, b_1, \ldots, b_d \rangle$,
the coordinate ring of the affine tangent cone at the singularity is
$S_d = \textnormal{gr}_{\mathfrak{m}}(R_d)$.
As $d - 2 > 2$, we have
$S_d = \C[q,Q,e,b_0, b_1, \ldots, b_d]$
with relations
$$
\left\lbrace \begin{array}{l}
eb_j = qb_{j+1} + Qb_{j-1} \\
b_j b_k - b_{j-1} b_{k+1} = 0
\end{array} \right.
\textnormal{for } 1 \leq j \leq k \leq d-1.
$$
We want to show this is an integral domain,
and we do so by showing it is a subring of another simpler integral domain.

Consider the algebra
$A := \C[q,Q,e,s,t] / \langle est = qt^2 + Qs^2 \rangle $,
the cyclic group $\mu_d = \langle g | g^d \rangle$,
and the action of $\mu_d$ on $A$ via
$$
g: q \mapsto q, Q \mapsto Q, e \mapsto e,
s \mapsto \zeta s, t \mapsto \zeta t
$$
for $\zeta = e^{\frac{2 \pi i}{d}}$ a primitive $d$-th root of unity.

Labelling $\widetilde{b_j} := s^{d-j}t^j$,
we can easily see
$A^{\mu_d} := \C[q,Q,e, \widetilde{b_0}, \widetilde{b_1}, \ldots, \widetilde{b_d}]$ with relations
$$
\left\lbrace \begin{array}{l}
e \widetilde{b_j} = q \widetilde{b_{j+1}} + Q \widetilde{b_{j-1}} \\
\widetilde{b_j} \widetilde{b_k} - \widetilde{b_{j-1}} \widetilde{b_{k+1}} = 0
\end{array} \right.
\textnormal{for } 1 \leq j \leq k \leq d-1.
$$
From this, it is easy to form the isomorphism $A^{\mu_d} \cong S_d$.

As $est - qt^2 - Qs^2$ is irreducible as a polynomial over $\C$, $A$ is an integral domain.
The ring of invariants $A^{\mu_d}$ is thus also an integral domain.
Therefore the coordinate algebra of the affine tangent cone, $S_d$, is an integral domain.
Hence, the projective tangent cone $\mathbb{P}T_0(\mathcal{Y}(d)) = $ Proj$(S_d)$ is irreducible, and in particular reduced.

\subsection{Toric hyper-Kähler varieties and Coulomb branches}
\label{toric}

Consider the family of symplectic singularities described in the paper by Namikawa, given as toric hyper-Kähler varieties (\cite{Namikawa}).
These all have a single isolated singularity.
We will only be interested in the 4-dimensional case,
but the results in this subsection work for arbitrary dimension,
and so will be given as such.

Toric hyper-Kähler varieties were introduced by Bielawski-Dancer (\cite{BielawskiDancer}).
Those we are referencing are constructed using the following data.
Let $a$, $b$ be positive integers such that $a \geq b + 2$.
Let $A$ (resp. $B$) be a $b \kern-2pt \times \kern-2pt a$-matrix
(resp. an $a \kern-2pt \times \kern-2pt (a-b)$-matrix)
with integer coefficients such that the following sequence is exact:
$$
0 \rightarrow \Z^{a-b} \xrightarrow{B} \Z^a \xrightarrow{A} \Z^b \rightarrow 0.
$$
Given this, a Hamiltonian action of $(\C^*)^b$ on the symplectic manifold $(\C^{2a}, \omega)$ is constructed,
with moment map $\mu: \C^{2a} \rightarrow \C^b$.
The toric hyper-Kähler variety $Y(A,0) = \mu^{-1}(0) / \kern-4pt / (\C^*)^b$
is then given as the GIT quotient of the preimage of $0$ under $\mu$, by the Hamiltonian action.

Those considered in Namikawa's paper have restrictions on $A$ which result in $Y(A,0)$ being an isolated symplectic singularity which is locally simply-connected.
However, we can consider a more general case.
As described by Namikawa (\cite[\S 5.ii]{NakajimaI}),
we can realise such a toric hyper-Kähler variety as a Coulomb branch of a 3-dimensional $\mathcal{N} = 4$ SUSY gauge theory with gauge group a torus.
Such a Coulomb branch is conical, so if it has isolated singularities,
it must have only one such, and it must be the origin.
To address this, we show that each of these has a reduced projective tangent cone at the origin.
These Coulomb branches have been given an explicit algebraic presentation (\cite[\S 4.i]{BFNII}), so this is the framework we will be using throughout this proof.

For a $2n$-dimensional Coulomb branch $C$, the torus gauge group $T$ must be rank $n$, so $T = (\C^{\times})^n$.
Let $Y = \Z^n$ denote the coweight lattice of $T$,
and $\mathfrak{t}$ the Lie algebra of $T$ with corresponding dual space $\mathfrak{t}^* \cong \C^n$.
The Coulomb branch also depends on a finite-dimensional representation $\mathbf{N}$ of $T$.
Choosing a basis of weight vectors for $\mathbf{N}$,
this amounts to a collection of characters $\xi_1, \kern1pt \ldots, \kern1pt \xi_m$,
viewed as elements of $\mathfrak{t}^*$.
Taking $x_1, \ldots, x_n \in \mathfrak{t}^*$ such that
$x_i(\lambda_1, \ldots, \lambda_n) = \lambda_i$,
we have that each character $\xi_i \in \textnormal{span}_{\Z} \{x_1, \ldots, x_n\} \cong \Z^n$.

To convert the previous data of a toric hyper-Kähler variety $Y(A,0)$ to the data of a Coulomb branch,
take the gauge group to be the torus $T = (\C^{\times})^{a-b}$,
and the representation of $T$ to be $\mathbf{N} = \C^a$,
considered as a representation of $T$ through $B$.
Specifically, the characters $\xi_1, \kern1pt \ldots, \kern1pt \xi_a \in \Z^{a-b}$ of $\mathbf{N}$ in this consideration are given by rows of $B$.

The coordinate algebra $R$ of $C$ is generated over $\C$ as 
$\C[x_1, \ldots, x_n, \{ r^{\lambda} : \lambda \in \Z^n \}]$.
Defining a function $d$ on $\Z^2$ as
$$
d(k,l) = \left\lbrace
\begin{array}{l}
0 \\
\min(|k|, |l|)
\end{array}
\begin{array}{l}
\textnormal{if $k$ and $l$ have the same sign,} \\
\textnormal{if $k$ and $l$ have different signs,}
\end{array} \right.
$$
we know the relations for $R$ are
$$
r^0 = 1, \kern5pt
r^{\lambda} r^{\mu} = \prod^m_{i=1} {\xi_i}^{d(\xi_i(\lambda), \xi_i(\mu))} r^{\lambda + \mu}.
$$

There are $n$ $\Z$-gradings of $R$, coming from $x_1, \ldots, x_n$ having degree 0
and $r^{\lambda}$ having degree equal to the $n$th coordinate of $\lambda$.

\begin{proposition}
The projective tangent cone of the Coulomb branch $C$ at $0$ is reduced.
\end{proposition}
\begin{proof}

Let $\mathfrak{m} = \langle x_1, \ldots, x_n, \{ r^{\lambda} : \lambda \in \Z^n \} \rangle$ be the maximal ideal at $0$,
$\textnormal{gr}_{\mathfrak{m}}(R)$ be the coordinate ring of the affine tangent cone of $R$ at $0$,
and suppose $f \in \textnormal{gr}_{\mathfrak{m}}(R)$ is nilpotent.

The $\Z$-gradings of $R$ induce $\Z$-gradings on $\textnormal{gr}_{\mathfrak{m}}(R)$,
with the generators in the tangent cone taking the same degrees as they did in the original ring.

Take $k \in \N$ such that $f^k = 0$.
Fix a choice of one $\Z$-grading,
and let $l \in \Z$ be maximal such that $f_l$, the $l$-degree part of $f$, is non-zero.
Considering the expansion of $f^k = (\sum_{i \in \N} f_i)^k = 0$,
the maximum of the degrees of the terms involved is $lk$, of which there is only ${f_l}^k$.
Therefore ${f_l}^k = 0$.
Arguing by induction on $l$, without loss of generality we can assume $f$ is homogeneous.
Therefore, we can assume $f$ is homogeneous with respect to all of the $\Z$-gradings.

Using this and the relations of $R$, we can assume further that $f \in r^{\lambda} \cdot \textnormal{Sym} (t^*)$ for some $\lambda$.
Then, the condition $f^k = 0$ would give a non-trivial relation between $x_1, \ldots, x_n$ and $r^{\lambda}$,
so to find a contradiction it will suffice to prove that $x_1, \ldots, x_n$ and $r^{\lambda}$ freely generate a polynomial subalgebra of $R$.

Consider the ideal $I = \langle x_1, \ldots, x_n \rangle$ of $\textnormal{gr}_{\mathfrak{m}}(R)$, and the associated graded algebra of $\textnormal{gr}_{\mathfrak{m}}(R)$ with respect to $I$, $S := \textnormal{gr}_I(\textnormal{gr}_{\mathfrak{m}}(R))$.
Looking at the relations of $R$,
it is easy to deduce that $S$ is generated over $\C$ as $\C[x_1, \ldots, x_n, \{ r^{\lambda} : \lambda \in \Z^2 \}]$,
with every relation in one of the following forms:
$$
r^0 = 1, \kern5pt
r^{\lambda} r^{\mu} = 0, \kern5pt
r^{\lambda + \mu} = 0.
$$
Since there is no term involving $x_1, \ldots, x_n$ in the relations of $S$,
it is clear that they freely generate their subalgebra in $S$,
so we cannot have $f \in$ Sym$(t^*)$ nilpotent.
Therefore we must have $\lambda \neq 0$.

If we had $r^{\lambda} = 0$ in $S$, this would come from a relation of the third form.
That is, it would be given by a relation $r^{\mu_1} r^{\mu_2} = r^{\lambda}$ in $R$.
This would become $r^{\lambda} = 0$ in $\textnormal{gr}_{\mathfrak{m}}(R)$ however, giving the contradiction that $f = 0$ in $\textnormal{gr}_{\mathfrak{m}}(R)$.
Therefore $r^{\lambda} \neq 0$ in $S$.

As $d(\xi_i(\lambda), \xi_i(\lambda)) = 0$ for all $i$,
we observe the relation $(r^{\lambda})^k = r^{k \lambda}$ in $R$,
which gives the relation $r^{k \lambda} = 0$ in $S$.
Notably, it is observable that $r^{\lambda}$ cannot have a relation on its own, and so freely generates its subalgebra of $S$.
As there is no term involving $x_1, \ldots, x_n$ in the relations of $S$,
it follows that $x_1, \ldots, x_n$ and $r^{\lambda}$ freely generate their subalgebra in $S$.

A relation between these variables in $\textnormal{gr}_{\mathfrak{m}}(R)$ would induce one in $S = \textnormal{gr}_I(\textnormal{gr}_{\mathfrak{m}}(R))$.
Therefore, $x_1, \ldots, x_n$ and $r^{\lambda}$ freely generate their subalgebra in $\textnormal{gr}_{\mathfrak{m}}(R)$.
Therefore $f^k = 0$ is a contradiction,
so $\textnormal{gr}_{\mathfrak{m}}(R)$ is indeed reduced,
and the projective tangent cone $\mathbb{P}T_0(C) = $ Proj$(\textnormal{gr}_{\mathfrak{m}}(R))$ is reduced also.
\end{proof}

\begin{remark}
For most choices of $\textbf{N}$, there will be sufficient characters such that we will have a relation in $R$ of the form
$$
r^{\lambda} r^{\mu} = \prod^k_{i=1} {\xi_{n_k}} r^{\lambda + \mu}
$$
for some non-redundant generators $r^{\lambda}$ and $r^{\mu}$ with $k \geq 2$,
giving $r^{\lambda} r^{\mu} = 0$ in $\textnormal{gr}_{\mathfrak{m}}(R)$.
This gives that the coordinate algebra is not an integral domain,
so the projective tangent cone is not irreducible.
Therefore, while the projective tangent cone at the singularity is reduced for all these varieties,
it is reducible in most cases.
This gives an interesting contrast with the previous case of the $\mathcal{Y}(d)$, in which the projective tangent cone is always irreducible.
The overlap between these cases is $a_2$, which has smooth projective tangent cone.
\end{remark}

\section{\texorpdfstring{$\Q$}{Q}-factorial terminalisations}
\label{QFT}

In accordance with the minimal model program,
we would like to find the $\QFl$ terminalisations of the new symplectic singularity \yogh.

We will not be able to consider being locally simply-connected and being $\QFl$ terminal simultaneously, however.
For example, consider the Slodowy slice of the closure of the nilpotent orbit of Lie algebra $E_7$ with Bala-Carter label $A_1 + A_4$,
which will be isomorphic to a quotient of $a_2$ (\cite{NilpOrbLocGeom}).
This is $\QFl$ terminal as the closure of this nilpotent orbit is known to be (\cite{NilpOrbQFTs}).
As $a_2$ itself is not $\QFl$, this shows that taking the universal cover of an isolated symplectic singularity can lose the property of being $\QFl$ terminal.
Additionally, there are locally simply-connected isolated singularities whose $\QFl$ terminalisations are not locally simply-connected, for example $\bar{h}_{2,3}$ (Appendix \ref{h_2c3}).
This shows that taking the $\QFT$ of a singularity can lose the property of being locally simply-connected.

In this section we find all 12 $\QFT$s of \yogh, prove they are exhaustive, and extend this proof to find some of the $\QFT$s for $\C^4/G_5$.
In all blow-ups being considered, the center is a (reduced) prime divisor.
Since this has no effect on smooth points
 we will not be concerned about the effect on the smooth affine charts,
and instead only consider those containing singularities.

\subsection{Crepant partial resolutions of the new singularity}
The pre-image in $X_{4,4}$ of the singular locus in $\C^4/G_5$ has four irreducible components.
These four non-$\Q$-Cartier divisors are contained in a single $\langle \phi, \tau, \chi \rangle$-orbit of six divisors.
One divisor in this orbit is given by the ideal $J := \langle \widetilde{a_0}, w_1, w_0, w_{-1}, t_2, t_{-2} \rangle$, with
$$
\begin{array}{c}
\widetilde{a_0} := \frac{1}{2}c_0 + \frac{1}{2}d_0 - \frac{1}{3}h^2, \\
w_1 := z_1 + (-\frac{1}{6}i\sqrt{3} - \frac{1}{2})hc_1
+ (-\frac{1}{6}i\sqrt{3} + \frac{1}{2})hd_1, \\
w_0 := z_0 + \frac{2}{9}i\sqrt{3}h^3, \\
w_{-1} := z_{-1} + (-\frac{1}{6}i\sqrt{3} + \frac{1}{2})hc_{-1}
+ (-\frac{1}{6}i\sqrt{3} - \frac{1}{2})hd_{-1}, \\
t_2 := hz_2 + (-\frac{1}{4}i\sqrt{3} + \frac{3}{4}){c_1}^2
- \frac{1}{2}i\sqrt{3}c_1d_1
+ (-\frac{1}{4}i\sqrt{3} - \frac{3}{4}){d_1}^2, \\
t_{-2} := hz_{-2} + (-\frac{1}{4}i\sqrt{3} - \frac{3}{4}){c_{-1}}^2
- \frac{1}{2}i\sqrt{3}c_{-1}d_{-1}
+ (-\frac{1}{4}i\sqrt{3} + \frac{3}{4}){d_{-1}}^2, \\
\end{array}
$$
Label the ideals in the orbit as follows:
$$
J_1 := J, \kern4pt
J_2 := \phi(J_1), \kern4pt
J_3 := \chi(J), \kern4pt
J_4 := \phi(J_3), \kern4pt
J_5 := \chi^2(J), \kern4pt
J_6 := \phi(J_5), \kern4pt
$$
The irreducible components of the preimage of the singular locus are given by the ideals $ J_3, J_4, J_5$ and $J_6$.

Blowing up at $J$, we can cover the blow-up with affine charts given by the projective coordinates corresponding to generators of $J$ being non-zero.
Amongst those, we can show that the affine charts corresponding to $\widetilde{a_0}, w_1, w_{-1}, t_2, t_{-2}$
cover the blow-up, so considering these will be sufficient for our purposes.

Of these, the $w_1, w_{-1}, t_2, t_{-2}$ affine charts turn out to be smooth,
leaving us with just the $\widetilde{a_0}$ affine chart to consider.

Denoting
$$
W_1 := \frac{w_1}{\widetilde{a_0}}, 
W_0 := \frac{w_0}{\widetilde{a_0}}, 
W_{-1} := \frac{w_{-1}}{\widetilde{a_0}}, 
T_2 := \frac{t_2}{\widetilde{a_0}}, 
T_{-2} := \frac{t_{-2}}{\widetilde{a_0}},
$$
we form an isomorphism from $\mathbb{C}[\bar{h}_{2,3}]$
(see Appendix \ref{h_2c3} for generators and relations)
to the coordinate algebra of this affine chart
by setting the images of the generators as such:

$$
\begin{array}{c}
\xi_1 \mapsto W_0 - \frac{4}{3} i \sqrt{3} h, \kern4pt
\xi_2 \mapsto 2W_0 + \frac{4}{3} i \sqrt{3} h, \kern4pt
\xi_3 \mapsto -W_0, \\
r^{(1,0)} \mapsto W_1, \kern4pt
r^{(-1,0)} \mapsto 2W_{-1}, \\

r^{(3,1)} \mapsto \frac{-2}{\sqrt{3}} T_2 + i {W_1}^2, \\
r^{(2,1)} \mapsto \frac{4}{\sqrt{3}} h W_1 + 2i W_1 W_0
+ (-2 \sqrt{3} + 2i) c_1 + (2 \sqrt{3} + 2i) d_1, \\
r^{(1,1)} \mapsto \frac{16}{3} i h^2 + 4i {W_0}^2
+ (-4 \sqrt{3} + 4i) c_0 + (4 \sqrt{3} + 4i) d_0, \\
r^{(0,1)} \mapsto \frac{-32}{\sqrt{3}} h W_{-1} + 8i W_0 W_{-1}
+ (-8 \sqrt{3} + 8i) c_{-1} + (8 \sqrt{3} + 8i) d_{-1}, \\

r^{(-3,-1)} \mapsto \frac{-2}{\sqrt{3}} T_{-2} + i {W_{-1}}^2, \\
r^{(-2,-1)} \mapsto \frac{2}{\sqrt{3}} h W_{-1} + i W_0 W_{-1}
+ (\sqrt{3} + i) c_{-1} + (-\sqrt{3} + i) d_{-1}, \\
r^{(-1,-1)} \mapsto \frac{4}{3} i h^2 + i {W_0}^2
+ (\sqrt{3} + i) c_0 + (-\sqrt{3} + i) d_0, \\
r^{(0,-1)} \mapsto \frac{-4}{\sqrt{3}} h W_1 + i W_1 W_0
+ (\sqrt{3} + i) c_1 + (-\sqrt{3} + i) d_1
\end{array}
$$

The singularity $\bar{h}_{2,3}$ is known to have precisely 2 $\QFT$s,
given by blowing up at
$\langle \xi_3, r^{(3,1)}, r^{(2,1)}, r^{(1,1)}, r^{(0,1)} \rangle$
or
$\langle \xi_3, r^{(-3,-1)}, r^{(-2,-1)}, r^{(-1,-1)}, r^{(0,-1)} \rangle$.
These consist of one copy of the isolated $\Q$-factorial symplectic singularity $\C^4 / \mu_3$,
and two other affine charts isomorphic to $\C^4$.
This describes the remaining steps required to reach a $\Q$-factorial terminalisation of \yogh.

Considering the ratios of projective coordinates giving generators of coordinate algebras of affine charts,
we get rational functions of $\C^4$ which are defined in a dense subset of the preimage of the origin of $\C^4$ in a given $\QFT$.
These are properties which can distinguish between them, and can be used to note that all 12 of these given $\QFT$s of \yogh$ $ are distinct.

For example, one of the $\QFT$s in the $J_3$ case has the fractions
$$
\frac{F_{0,4} G_{2,4}}{{G_{2,2}}^2},
\frac{{G_{2,2}}^2 \overline{G_{2,2}}}{{G_{2,4}}^3},
$$
which can be used to distinguish it from the other 11 cases.
The first can determine it as the result of the $J_3$ blow-up, and the second will determine it amongst the two further blow-ups of that.
The irreducibility of many of the polynomials involved helps with the proof.
By symmetry, this extends to showing all 12 are distinct.

\subsection{Position in the hierarchy of crepant partial resolutions}
\label{fig:hyperplane}
We want to see how this aligns with the structure of all possible crepant partial resolutions.
Symplectic singularities have hyperplane arrangements associated to them via the minimal model programme (\cite{NamHyperplane}, \cite{HyperplaneArr}).
This has the action of a Namikawa Weyl group $W$.
For a given $\QFT$ $\pi: Y \rightarrow X$, we consider the closure of the movable cone $\text{Mov}(Y/X)$, which is a fundamental domain for the $W$-action.
Crepant partial resolutions of $X$ are in bijection with the faces of $\overline{\text{Mov}(Y/X)}$. (\cite[Corollary 4.10]{PartResoln}).
A crepant blow-up of a given partial resolution will move from the corresponding face to one of higher dimension that contains the original in its closure.

Consider the data collected by Thiel for the hyperplane arrangement associated to $\C^4/G_5$ (\cite{ThielG5}).
The dimensions of the faces involved range from 0 to 4.
The crepant partial resolution with \yoghsp as its only singularity is obtained after two blow-ups,
and a $\QFT$ is given after two more,
so \yoghsp must correspond to a 2-dimensional face.
We want to know the possible number of hyperplanes intersecting such a face.
The following code was written in Sage (\cite{Sage}) to determine this from Thiel's data.

\begin{lstlisting}[language=Python]
load("G5.sage");
P = A.intersection_poset();
values = set();
for x in range(len(P)):
    if P.rank_function(x) == 2:
        values.add(len(P.lower_covers(x)));
print(values);
\end{lstlisting}

From this, we find that there are 2, 3, 4 or 6 such.
Given a specific number of hyperplanes intersecting a codimension 2 face,
we can take a 2-dimensional slice normal to the face and consider the open subsections in the slice to conclude that there will be twice as many distinct $\QFT$s of that partial resolution.
See the below diagram of such a slice for an example of how 3 intersecting hyperplanes leads to 6 open subsections.
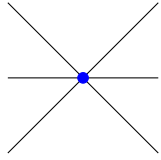
\begin{align*}
\begin{tikzpicture}
\draw (1,1) -- (-1,-1);
\draw (1,0) -- (-1,0);
\draw (1,-1) -- (-1,1);
\filldraw [blue] (0,0) circle (2pt);
\end{tikzpicture}
\end{align*}
Therefore, there are at most 12 4-dimensional faces containing a specific 2-dimensional face in their closure,
so there are at most 12 distinct $\QFl$ terminalisations of \yogh.
We conclude that the given list is exhaustive.

\begin{remark}
Note that we have only described 6 of the 12 crepant partial resolutions of \yogh given as 3-dimensional faces in the hyperplane arrangement,
but that this is sufficient to find all of the $\QFT$s.
\end{remark}

This property of having 12 $\QFT$s, and being two steps removed from each of them,
is somewhat exceptional amongst the currently known 4-dimensional, isolated, locally simply-connected symplectic singularities.
All other such which have had their $\QFl$ terminalisations described have only had 2 of them,
and they have only been one step removed from each.

$ $

The $\QFT$s of \yoghsp can used to find those of $\C^4/G_5$.
We know there are 92 such (\cite[Table 1]{HyperplaneArr}).
Recall that the blow-up of $\C^4/G_5$ at $I$ gave smooth affine charts other than $X_{13,1}$ and $X_{1,13}$ isomorphic to $a_2$, and $X_{4,4} =$ \yogh.

With $a_2$ as the closure of the minimal nilpotent orbit of $\mathfrak{sl}_3$,
it is well-known that it has precisely 2 $\QFT$s,
which also happen to be symplectic resolutions.

We can use $\langle \overline{G_{2,2}} \rangle \subseteq \C[\C^4/G_5]$ to form an ideal sheaf $\mathcal{I}$
which is principal for $X_{4,4}$:
$$\langle f_{4,4}, f_{9,3}, h - i\sqrt{3} f_{14,2} \rangle
\subseteq \C[X_{13,1}],
\biggl< \frac{R_{4,4}}{R_{1,13}}, \frac{{R_{3,9}}}{R_{1,13}},
h - i\sqrt{3} \frac{R_{2,14}}{R_{1,13}} \biggr> \subseteq \C[X_{1,13}].
$$
Blowing up at $\mathcal{I}$ leaves all other affine charts unaffected, including $X_{4,4}$, so can be followed on by the $\QFT$s of $X_{4,4}$ described above.

There is a symmetric ideal sheaf given by considering $\langle G_{2,2} \rangle \subseteq \C[\C^4/G_5]$ to get the complex conjugate $\bar{\mathcal{I}}$.
This similarly does not affect the $X_{4,4}$ affine chart, but it swaps the two $\QFT$s of the $X_{13,1}$ and $X_{1,13}$ affine charts with the other corresponding $\QFT$ of $a_2$.

Performing either of these blow-ups first allows these two simultaneous resolutions of the two $a_2$ singularities to combine with the 12 $\QFT$s of \yogh$ $ to give a total of 24 of the 92 $\Q$-factorialisations of $\C^4/G_5$.
A similar trick with the polynomials appearing in coordinate algebras can be used to show these are all distinct.

To summarise, among these crepant partial resolutions of $\C^4/G_5$,
we have found the new isolated symplectic singularity \yogh.
All of the others found were already known.
(The crepant partial resolutions of $\C^4/G_6$ have been fully described,
and found to contain no new isolated symplectic singularities (\cite{MyThesis}).)

\newpage

\appendix

\section{Relations for \texorpdfstring{$\C[\C^4/G_5]$}{C[C\^4/G\_5]}}
\label{G5Relns}

The 35 relations for the generators of $\C[\C^4/G_5]$ are as follows:

$$
\begin{array}{c}

2A_1A_{-1} - 2{A_0}^2 + 9h^2B_0, \\

A_1B_0 - A_0B_1 + 3hC_1, \\
A_{-1}B_0 - A_0B_{-1} - 3hC_{-1}, \\
A_1B_{-1} - A_0B_0 + 3hC_0 + 2h^4A_0, \\
A_{-1}B_1 - A_0B_0 - 3hC_0 + 2h^4A_0, \\

2B_1B_{-1} - 2{B_0}^2 - 6h^2{A_0}^2
- h^4B_0, \\

6A_1C_1 - 6A_0C_2 + 9h{B_1}^2
- h^3{A_1}^2, \\
6A_{-1}C_{-1} - 6A_0C_{-2} - 9h{B_{-1}}^2
+ h^3{A_{-1}}^2, \\
6A_1C_0 - 6A_0C_1 + 9hB_1B_0
- 4h^3A_1A_0, \\
6A_{-1}C_0 - 6A_0C_{-1} - 9hB_{-1}B_0
+ 4h^3A_{-1}A_0, \\
6A_1C_{-1} - 6A_0C_0 + 9h{B_0}^2
- 4h^3{A_0}^2, \\
6A_{-1}C_1 - 6A_0C_0 - 9h{B_0}^2
+ 4h^3{A_0}^2, \\
6A_1C_{-2} - 6A_0C_{-1} + 9hB_0B_{-1}
- 28h^3A_0A_{-1} - 36h^5B_{-1}, \\
6A_{-1}C_2 - 6A_0C_1 - 9hB_0B_1
+ 28h^3A_0A_1 + 36h^5B_1, \\

3B_1C_1 - 3B_0C_2 + 3h{A_1}^2A_0
+ 4h^3A_1B_1, \\
3B_{-1}C_{-1} - 3B_0C_{-2} - 3h{A_{-1}}^2A_0
- 4h^3A_{-1}B_{-1}, \\
6B_1C_0 - 6B_0C_1 + 6hA_1{A_0}^2
+ 4h^3A_0B_1 + h^3A_1B_0, \\
6B_{-1}C_0 - 6B_0C_{-1} - 6hA_{-1}{A_0}^2
- 4h^3A_0B_{-1} - h^3A_{-1}B_0, \\
6B_1C_{-1} - 6B_0C_0 + 6h{A_0}^3
+ 5h^3A_0B_0, \\
6B_{-1}C_1 - 6B_0C_0 - 6h{A_0}^3
- 5h^3A_0B_0, \\
6B_1C_{-2} - 6B_0C_{-1} + 6h{A_0}^2A_{-1}
+ 35h^3A_{-1}B_0 - 84h^4C_{-1}
+ 4h^7A_{-1}, \\
6B_{-1}C_2 - 6B_0C_1 - 6h{A_0}^2A_1
- 35h^3A_1B_0 - 84h^4C_1
- 4h^7A_1, \\

36C_2C_0 - 36{C_1}^2 + 54h^2A_1A_0B_1
- 18h^3C_2A_0 + 63h^4{B_1}^2
+ h^6{A_1}^2, \\
36C_{-2}C_0 - 36{C_{-1}}^2 + 54h^2A_{-1}A_0B_{-1}
+ 18h^3C_{-2}A_0 + 63h^4{B_{-1}}^2
+ h^6{A_{-1}}^2, \\
9C_2C_{-1} - 9C_1C_0 + 27h^2A_1A_0B_0
+ 36h^3C_1A_0 + 18h^4B_1B_0
+ 2h^6A_1A_0, \\
9C_{-2}C_1 - 9C_{-1}C_0 + 27h^2A_{-1}A_0B_0
- 36h^3C_{-1}A_0 + 18h^4B_{-1}B_0
+ 2h^6A_{-1}A_0, \\
2C_2C_{-2} - 2C_1C_{-1} + 9h^2A_1A_{-1}B_0
+ 48h^4{B_0}^2 + 32h^6A_1A_{-1}
+ 132h^8B_0 - 8h^{12}, \\
18C_1C_{-1} - 18{C_0}^2 + 27h^2A_1A_{-1}B_0
+ 126h^4{B_0}^2 + 8h^6{A_0}^2, \\

6C_2C_1 - 3{B_1}^3 + 2{A_1}^3A_0
+ 3h^2{A_1}^2B_1, \\
6C_{-2}C_{-1} - 3{B_{-1}}^3 + 2{A_{-1}}^3A_0
+ 3h^2{A_{-1}}^2B_{-1}, \\
6{C_1}^2 - 3{B_1}^2B_0 + 2{A_1}^3A_{-1}
+ 12h^2A_1A_0B_1 - 28h^3A_1C_1, \\
6{C_{-1}}^2 - 3{B_{-1}}^2B_0 + 2{A_{-1}}^3A_1
+ 12h^2A_{-1}A_0B_{-1} + 28h^3A_{-1}C_{-1}, \\
6C_1C_0 - 3B_1{B_0}^2 + 2A_1{A_0}^3 
+ 3h^2A_1A_0B_0 + 4h^3A_0C_1, \\
6C_{-1}C_0 - 3B_{-1}{B_0}^2 + 2A_{-1}{A_0}^3 
+ 3h^2A_{-1}A_0B_0 - 4h^3A_0C_{-1}, \\

18{C_0}^2 - 9B_1B_0B_{-1} + 6{A_0}^4
+ 9h^2{A_0}^2B_0 - 8h^6{A_0}^2.

\end{array}
$$

\newpage

\section{Use of \texorpdfstring{$\rho$}{ρ} to prove relations are complete}
\label{rhoAppendix}

We have that, for all $i$,
$$
\rho(h) = 0, \kern4pt
\rho(A_i) = 6F_{6,0} = \nu_A, \kern4pt
\rho(B_i) = 2F_{4,0}\overline{F_{4,0}} = \nu_B, \kern4pt
\rho(C_i) = {F_{4,0}}^3 + \overline{F_{4,0}}^3 = \nu_C.
$$
Consider the commuting diagram
$$
\begin{array}{ccc}
\textnormal{Free}(h, A_1, \ldots, C_{-2}) &

\xlongrightarrow{\pi_1} & \C[\C^4/G_5] \\ \\

\overline{\rho} \big\downarrow & & \big\downarrow \rho \\ \\

\textnormal{Free} (\nu_A, \nu_B, \nu_C) &
\xrightarrow{\pi_2} & \C[x,y]
\end{array}
$$

We want to find $\ker (\pi_1)$,
to find the list of relations of $\C^4/G_5$.

Since $\ker (\pi_1)$ will be generated by polynomials which are bi-homogeneous in the inherited bidegree grading,
we can replace all ideals $I$ involved in the calculations with their bi-homogeneous parts $I_{\bihom}$.

It is manually calculable that
$$ \ker (\rho)_{\bihom} = \langle h \rangle_{\bihom}, \kern5pt
\ker (\pi_2) = \langle
2{\nu_A}^4 - 3{\nu_B}^3 + 6{\nu_C}^2
\rangle .$$

We know
$$
\begin{array}{rl}

\ker (\pi_1) \subseteq & \kern-5pt
\ker (\rho \circ \pi_1)_{\bihom} =
\ker (\pi_2 \circ \overline{\rho})_{\bihom} = 
\overline{\rho}^{-1} (\ker (\pi_2))_{\bihom}
\end{array}
$$

$ \ker (\overline{\rho})_{\bihom} =
\langle h, A_1 A_{-1} - {A_0}^2, A_1 B_0 - A_0 B_1,
\textnormal{etc.} \rangle_{\bihom} $,
so
$$
\begin{array}{rl} \kern-10pt
\overline{\rho}^{-1} (\ker (\pi_2))_{\bihom} = \kern-4pt &
\overline{\rho}^{-1} (\langle
2{\nu_A}^4 - 3{\nu_B}^3 + 6{\nu_C}^2 \rangle)_{\bihom} \\
= \kern-4pt & \langle h, A_1 A_{-1} - {A_0}^2, A_1 B_0 - A_0 B_1, \ldots, \\
& 2{A_1}^3 A_0 - 3{B_1}^3 + 6 C_2 C_1,
2{A_1}^3 A_{-1} - 3{B_1}^2 B_0 + 6 C_1^2,
\ldots \rangle_{\bihom}.
\end{array}
$$

We've already found elements within
$\langle h \rangle_{\bihom}$
that can be added on to each of these other generators of
$\overline{\rho}^{-1} (\ker (\pi_2))_{\bihom}$
to make elements of $\ker (\pi_1)$.
Removing $h$ as a generator, label
$$
\begin{array}{l} \kern-10pt
I = \langle A_1 A_{-1} - {A_0}^2, A_1 B_0 - A_0 B_1, \ldots, \\
2{A_1}^3 A_0 - 3{B_1}^3 + 6 C_2 C_1,
2{A_1}^3 A_{-1} - 3{B_1}^2 B_0 + 6 C_1^2,
\textnormal{etc.} \rangle_{\bihom}.
\end{array}
$$
The following lemma will allow us to use this to justify having the full set of relations.

\begin{lemma} Let $R$ be a $\Z_{\geq 0}$-graded ring,
$f \in R$ an element of positive degree,
$I = \langle g_1, \ldots, g_n \rangle$ a homogeneous ideal of R,
$J \subseteq I + \langle f \rangle$ a prime ideal of $R$ with $f \not\in J$,
and $h_1, \ldots, h_n$ elements of $R$ such that $\forall i, g_i + fh_i \in J$.
Then $J = \langle g_1 + fh_1, \ldots, g_n + fh_n \rangle$.
\end{lemma}
\begin{proof}
Let $J' = \langle g_1 + fh_1, \ldots, g_n + fh_n \rangle$.
Clearly $J' \subseteq J$.
Suppose for a contradiction that $\exists x \in J \setminus J'$.
By using $J' \subseteq J$ to replace $g_i$ with $-fh_i$ in the expression of $x \in I + \langle f \rangle$,
without loss of generality we can assume $x \in \langle f \rangle$.
Let $y = \frac{x}{f} \in R$.
$fy \in J$ and $f \not\in J$, so by primality of $J$, $y \in J$.
$x = yf$ and $x \not\in J'$ so $y \not\in J'$.
$y$ is of strictly lower degree than $x$, so induction can be used to give a contradiction that there must be a negative degree element of $J$.
\end{proof}

We can make use of this lemma by taking
$$
R = \textnormal{Free}(h, A_1, \ldots, C_{-2}), \kern5pt
f = h, \kern5pt
I \textnormal{ as above,} \kern5pt
J = \ker (\pi_1).
$$
This gives that the polynomials generating
$I$,
once adapted by addition of an element of $\langle h \rangle$,
give a full set of generators of $J$,
i.e. a full set of relations of $\C^4/G_5$.

This applies well in this case due to the natural grading of $\C[\C^4/G_5]$,
and the codomain of $\overline{\rho}$ being able to be chosen to be of only one Krull dimension above that of $\C[x,y]$ (which allows for $\ker (\pi_2)$ to have only one generator).

We can also use this justification for the relations of $\C[$\yogh$]$,
due to the grading of the total degree still being non-negative,
and the existence of the $h_i$ is given by the fact that $h$ is not contained in the prime ideal $I$ being blown-up at to give $S_I \supseteq$ \yogh.

However, the grading coming from the total degree is not guaranteed to be non-negative in some of the coordinate algebras of the affine charts, for example $X_{9,3}$.

\newpage

\section{Coordinate algebra for \texorpdfstring{$\bar{h}_{2,3}$}{h\_2,3} and \texorpdfstring{$\Q$}{Q}-factorial terminalisations}
\label{h_2c3}

The quasi-minimal singularity $ag_2$ appears originally as an elementary slice in the affine Grassmannian Gr$_{G_2}$ (\cite[\S 6.3.2]{QuasiMinOrig}).
It can also be given as $\bar{h}_{2,3}$, and realised as the Coulomb branch of a quiver gauge theory with gauge group $(\C^\times)^2$ (\cite[p.33]{AffGrass}):

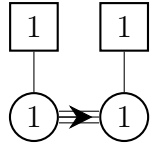
\begin{center}

$ $

\begin{tikzpicture}
[gauge/.style = {circle, draw, thick, inner sep = 3pt},
flavour/.style = {rectangle ,draw, thick, minimum size=18pt},
scale = 1.2]

\node[gauge] (v1) at (0,1) {$1$};
\node[flavour] (w1) at (0,2) {$1$};
\draw (v1) -- (w1);

\node[gauge] (v2) at (1,1) {$1$};
\node[flavour] (w2) at (1,2) {$1$};
\draw (v2) -- (w2);

\draw[double distance = 4pt] (v1) -- (v2);
\draw (v1) -- (v2);
\draw[-{Stealth[width = 9pt, length = 9pt, inset = 3pt]}] (v1) -- (0.65,1);

\end{tikzpicture}
\end{center}

This is a non-simply laced quiver, and so the Coulomb branch can be given as the fixed point subvariety of the Coulomb branch of a corresponding simply laced quiver under the automorphism $\sigma$ (\cite[\S 4]{BFNAffGrass}):

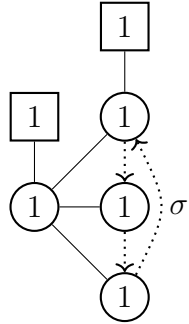
\begin{center}

$ $

\begin{tikzpicture}
[gauge/.style = {circle, draw, thick, inner sep = 3pt},
flavour/.style = {rectangle ,draw, thick, minimum size=18pt},
scale = 1.2]

\node[gauge] (v1) at (0,1) {$1$};
\node[flavour] (w1) at (0,2) {$1$};
\draw (v1) -- (w1);

\node[gauge] (v2) at (1,2) {$1$};
\node[gauge] (v3) at (1,1) {$1$};
\node[gauge] (v4) at (1,0) {$1$};
\node[flavour] (w2) at (1,3) {$1$};
\draw (v1) -- (v2);
\draw (v1) -- (v3);
\draw (v1) -- (v4);
\draw (v2) -- (w2);

\draw[->, dotted, thick] (v2) -- (v3);
\draw[->, dotted, thick] (v3) -- (v4);
\draw[->, dotted, thick] (v4) .. controls (1.5,1) .. (v2);
\node at (1.6,1) {$\sigma$};

\end{tikzpicture}
\end{center}

In the context of the BFN framework for Coulomb branches of affine quiver gauge theories (\cite[\S 4]{BFNII}), the latter can be given by as a representation of $(\C^\times)^4$ with characters
$$ (-1,0,0,0), (1,-1,0,0), (1,0,-1,0), (1,0,0,-1), (0,1,0,0); $$
and so the former can be calculated to be given by a representation of $(\C^\times)^2$ with characters
$$ (-1,0), (1,-3), (0,1). $$
Using the link between the toric hyper-Kähler constructions (\ref{toric}), we can take the matrix $A = (1,1,3)$.
Therefore this is one of the varieties considered in \cite{Namikawa},
and so this is known to be isolated and locally simply-connected.

For the BFN presentation, the generators of the coordinate algebra reduce down to
$$
\C[\bar{h}_{2,3}] = \C[x, y, r^{(1,0)}, r^{(-1,0)}, r^{(3,1)}, r^{(2,1)}, r^{(1,1)}, r^{(0,1)}, r^{(-3,-1)}, r^{(-2,-1)}, r^{(-1,-1)}, r^{(0,-1)}].
$$
Note that within this, we have the characters specified as
$$
\xi_1 = -x, \kern10pt \xi_2 = x - 3y, \kern10pt \xi_3 = y.
$$
Examples of relations include
$$
r^{(3,1)} r^{(0,-1)} = y (r^{(1,0)})^3, \kern10pt
r^{(3,1)} r^{(0,1)} = r^{(2,1)} r^{(1,1)}.
$$

By the discussion in \cite{Namikawa},
this singularity has 2 $\QFT$s.
In our notation, they arise by blowing up at the ideals
$$\langle y, r^{(3,1)}, r^{(2,1)}, r^{(1,1)}, r^{(0,1)} \rangle
\text{ and } \langle y, r^{(-3,-1)}, r^{(-2,-1)}, r^{(-1,-1)}, r^{(0,-1)} \rangle.$$
It can easily be shown that either blow-up has a unique singular point,
locally isomorphic to $\C^4/\mu_3$.

\begin{remark}
We can generalise this analysis to the full collection of 4-dimensional singularities $\bar{h}_{2,k}$, as considered in (\cite[\S 3.3]{AffGrass}).
Note the singularities $h_{2,k}$ also listed are simply quotients of $\C^{4}$ by finite groups, and so not of significant interest here.
Forming $\bar{h}_{2,k}$ as a fixed point subvariety of a Coulomb branch with a simply laced quiver, with method similar to that above,
we can identify $\bar{h}_{2,k}$ of the singularities studied by Namikawa.
Hence, we conclude that these singularities are isolated, locally simply-connected, and have 2 $\QFT$s.
\end{remark}

\newpage

\bibliography{ScholarRefs}
\bibliographystyle{abbrv}

\end{document}